\documentclass[12pt]{amsart}

\usepackage{pdfsync}         % pdf inverse search
\usepackage{enumerate}
\usepackage{comment}
\usepackage{amssymb}
\usepackage{subcaption}
\usepackage{mathtools}
\usepackage[OT1]{fontenc}    %
\usepackage{graphicx}
\usepackage{mathrsfs}

\newenvironment{labeledlist}[2][\unskip]
{ \renewcommand{\theenumi}{\textnormal{({#2}\arabic{enumi}{#1})}}
  \renewcommand{\labelenumi}{\textnormal{({#2}\arabic{enumi}{#1})}}
  \begin{enumerate} }
{ \end{enumerate} }

\usepackage{palatino}
\usepackage[left=3cm,top=3.8cm,right=3cm]{geometry}        % for a good layout

\usepackage{xcolor}
\usepackage[pagebackref]{hyperref}
\hypersetup{
   colorlinks,
    linkcolor={red!60!black},
    citecolor={blue!60!black},
    urlcolor={blue!90!black}
}

\numberwithin{equation}{section}

\theoremstyle{plain}
\newtheorem{theorem}{Theorem}[section]
\newtheorem{corollary}[theorem]{Corollary}
\newtheorem{prop}[theorem]{Proposition}
\newtheorem{lemma}[theorem]{Lemma}

\theoremstyle{remark}
\newtheorem{remark}[theorem]{Remark}

\theoremstyle{definition}
\newtheorem{definition}[theorem]{Definition}

\newcommand{\1}{\mathbf{1}}

\newcommand{\e}{\varepsilon}
\newcommand{\N}{\mathbb{N}}
\newcommand{\R}{\mathbb{R}}
\newcommand{\Z}{\mathbb{Z}}

\newcommand{\cP}{\mathcal{P}}

\newcommand{\cD}{\mathcal{D}}

\newcommand{\cE}{\mathcal{E}}

\newcommand{\cQ}{\mathcal{Q}}

\DeclareMathOperator{\supp}{supp}

\title[Inverse theorems in $\R^d$]{Inverse theorems for discretized sums and $L^q$ norms of convolutions in $\mathbb{R}^d$}

\author{Pablo Shmerkin}

\address{Department of Mathematics, University of British Columbia}
\email{pshmerkin@math.ubc.ca}
\urladdr{http://www.pabloshmerkin.org}
\thanks{Partially supported by an NSERC Discovery Grant}


\begin{document}

\begin{abstract}
  We prove inverse theorems for the size of sumsets and the $L^q$ norms of convolutions in the discretized setting, extending to arbitrary dimension an earlier result of the author in the line. These results have applications to the dimensions of dynamical self-similar sets and measures, and to the higher dimensional fractal uncertainty principle. The proofs are based on a structure theorem for the entropy of convolution powers due to M.~Hochman.
\end{abstract}

\maketitle

\section{Introduction and main results}

In order to state and discuss our main results, we introduce some notation.  We let $\cD_k$ be the family of half-open dyadic cubes of side length $2^{-k}$ in $\R^d$ (the value of $d$ will be implicit from context). For $x\in \R^d$, the only element of $\cD_k$ containing $x$ is denoted $\cD_k(x)$. More generally, if $A\subset\R^d$, then $\cD_k(A)$ denotes the family of cubes in $\cD_k$ intersecting $A$, and we denote $|A|_{2^{-k}}=|\cD_k(A)|$.

We write $[S]=\{0,1,\ldots,S-1\}$ for short; if $S$ is not an integer, then $[S]:=[\lfloor S\rfloor]$. Logarithms are always to base $2$.

\begin{definition}
  Given $m\in\N$, a \textbf{$2^{-m}$-set}  is a subset of $2^{-m}\Z^d\cap [0,1)^d$, and a \textbf{$2^{-m}$-measure} is a probability measure supported on $2^{-m}\Z^d\cap [0,1)^d$.

  A $2^{-SL}$-set is called \textbf{$(L,S)$-uniform} if there is a sequence $(R_s)_{s\in[S]}$ such that for each $s$ and each $I\in\cD_{sL}(A)$ it holds that
  \[
    |A\cap I|_{2^{-(s+1)L}}=R_s.
  \]

  In this case, we refer to $R_s$ as the \textbf{branching numbers} of $A$, and we also say that $A$ is $(L,(R_s)_{s\in[S]})$-uniform.
\end{definition}

Given a finitely supported measure $\mu$ and $q\in [1,+\infty)$, we define the (discrete) $L^q$ norm of $\mu$ as
\begin{equation} \label{eq:def-q-norm}
  \|\mu\|_q^q = \sum_x \mu(x)^q.
\end{equation}
We denote the Grassmanians of linear and affine $k$-dimensional planes in $\R^d$ by $\mathbb{G}(d,k)$ and $\mathbb{A}(d,k)$, respectively. Let $\pi_W$ denote the orthogonal projection onto a subspace $W\leq \R^d$. If $c>0$ and $I$ is a cube, we denote by $c I$ the cube with the same center as $I$ and side length $c$ times the side length of $I$.

We can now state our main result, which has informally been announced in \cite[\S 3.8.3]{Shmerkin23}.
\begin{theorem} \label{thm:inverse}
  For each $q\in (1,\infty)$ and $\delta>0$ the following holds if $L\ge L(\delta)\in\N$ and $\e<\e(L)$,  for all sufficiently large $S=S(\delta,L,\e)$:

  Let $m=SL$. Suppose $\mu,\nu$ are $2^{-m}$-measures on $[0,1)^d$ such that
  \[
    \|\mu*\nu\|_q \ge 2^{-\e m} \|\mu\|_q.
  \]
  Then there exist $2^{-m}$-sets $A\subset \supp\mu$ and $B\subset\supp\nu$ and a sequence $(k_s)_{s\in [S]}$ taking values in $[0,d]$ such that:
  \begin{labeledlist}{A}
    \item\label{A1} $\| \mu|_A \|_q \geq 2^{-\delta m}\|\mu\|_q$.
    \item\label{A2} $\mu(x) \leq 2 \mu(y)$ for all $x,y\in A$.
    \item\label{A3} $A$ is $(L,(R_s)_{s\in [S]})$-uniform for some sequence $(R_s)$.
    \item\label{A4} For each $s\in [S]$ and each $I\in\cD_{sL}(A)$, there is $W_I\in\mathbb{G}(d,d-k_s)$ such that
    \[
      R_s \ge 2^{(k_s-\delta) L}  \left|\pi_{W_I}(A\cap I)\right|_{2^{-(s+1)L}}.
    \]
  \end{labeledlist}
  \begin{labeledlist}{B}
    \item\label{B1} $\| \nu|_B \|_1 = \nu(B) \geq 2^{-\delta m}$.
    \item\label{B2} $\nu(x) \leq 2 \nu(y)$ for all $x,y\in B$.
    \item\label{B3} $B$ is $(L,S)$-uniform.
    \item\label{B4} For each $s\in [S]$ and each $I\in\cD_{sL}(B)$, there is $V_I\in\mathbb{A}(d,k_s)$ such that
    \[
      I\cap B \subset \mathcal{D}_{(s+1)L}(V_I).
    \]
  \end{labeledlist}

  Moreover, if $A\cup B\subset [1/3,2/3)^d$ then, after translating $A$ and $B$ by suitable vectors in $[-1/3,1/3)^d$, we may also assume that
  \begin{labeledlist}{C}
    \item\label{C} $x\in \frac{1}{3}\mathcal{D}_{sL}(x)$ for each $x\in A\cup B$ and $s\in [S]$.
  \end{labeledlist}

\end{theorem}

We make some remarks on the meaning and novelty of this statement.

\begin{remark}
  Condition \ref{A4} says that $A\cap I$ is densely contained in the product of a $k_s$-dimensional subspace and another (arbitrary) set - we refer to this behavior as ``almost saturation''. On the other hand, \ref{B4} states that $B\cap I$ is contained in a $k_s$-dimensional plane at the appropriate scale. Thus, Theorem \ref{thm:inverse} can be roughly summarized as saying that if no $L^q$ flattening occurs in $\mu*\nu$ then, after passing to regular subsets $A, B$ that capture, respectively, a ``large'' proportion of the $L^q$ norm and the mass, at each scale $2^{-sL}$ the measure $\mu$ is ``almost saturated'' on a $k_s$-dimensional subspace, while $\nu$ is ``almost contained'' in a $k_s$-dimensional affine plane.
\end{remark}

\begin{remark}
  The one dimensional version of this result is \cite[Theorem 2.1]{Shmerkin19}. In the case $d=1$, the only possibilities are $k_s=0$ and $k_s=1$, corresponding to $B$ having ``no branching'' and $A$ having ``almost full branching'', respectively. In higher dimensions, the possibility that $0<k_s<d$ makes the situation more complicated. Lack of smoothing can indeed happen due to saturation on intermediate-dimensional subspaces. For example if $\mu=\nu$ is the uniform measure $\lambda_{k}$ on some subspace of dimension $k\in [1,d-1]$ or, more generally, if $\mu$ is (roughly) a product measure $\lambda_{k}\times \mu'$, where $\mu'$ is arbitrary, and $\nu$ is concentrated on $\supp\lambda_{k}$. It is also possible to construct examples where the conclusion of Theorem \ref{thm:inverse} fails, and $\mu=\nu$ is roughly the uniform measure on a plane  at each scale, with the dimension of the plane varying according to the scale; see the discussion \cite[p.338]{Shmerkin19}, which extends to higher dimensions in a straightforward manner. This can be extended to show that \ref{A4} and \ref{B4} are necessary in Theorem \ref{thm:inverse} even for $\mu$ and $\nu$ of very different sizes. In this sense, Theorem \ref{thm:inverse} is conceptually fairly tight, but we underline that we are making no claims on how the different subspaces are related to each other - one would expect that in some rather strong sense, all the planes arising at a given scale are ``essentially the same'' up to translation, but are currently unable to prove this.
\end{remark}

Theorem \ref{thm:inverse} can be seen as an $L^q$ analog of Hochman's higher dimensional inverse theorem for entropy \cite[Theorem 2.8]{Hochman17}. Indeed, our proof is based on Hochman's machinery, but we do not apply \cite[Theorem 2.8]{Hochman17} directly. Instead, we appeal to \cite[Theorem 4.12]{Hochman17} (recalled as Theorem \ref{thm:hochman-conv-saturated} below) to prove the following inverse theorem for sumsets first. We denote the $r$-neighborhood of a set $X$ by $X(r)$.

\begin{theorem} \label{thm:inverse-thm-A+A}
  For each $\delta>0$ the following holds if $L\ge L_0(\delta)\in\N$, $0<\sigma<\sigma(L)$,  and $S\ge S_0(\delta,L,\sigma)$:

  Let $m=SL$. Let $A$ be a  $2^{-m}$ set in $[0,1)^d$ such that
  \[
    |A+A|\le 2^{\sigma m}|A|.
  \]
  Then there exists a subset $A'\subset A$ such that the following holds:
  \begin{enumerate}[\rm(i)]
    \item \label{it:inverse-A+A-i} $A'$ is $(L,(R_s)_{s\in[S]})$-uniform for some sequence $(R_s)_{s\in [S]}$.
          \item\label{it:inverse-A+A-ii}  $|A'|\ge 2^{-\delta m}|A|$.
    \item \label{it:inverse-A+A-iii} For each $s\in [S]$, there is $k_s\in \{0,1,\ldots, d\}$  such that
          \[
            \log R_s \ge L(k_s-\delta),
          \]
          and for each $I\in\cD_{sL}(A')$ there is a $k_s$-dimensional affine plane $W_I$ such that
          \[
            A'\cap I\subset  W_{I}\big(2^{-(s+1)L}\big).
          \]
  \end{enumerate}
\end{theorem}

\begin{remark}
  Theorem \ref{thm:inverse-thm-A+A} is closely related to Freiman's Theorem in additive combinatorics \cite[Theorem 5.32]{TaoVu06}, which roughly states that a set with small sumset has large intersection with a generalized arithmetic progression. However, even the strongest known quantitative version of Freiman's Theorem \cite[Theorem 1.4]{Sanders13} does not seem, by itself, to be sufficient for our purposes.
\end{remark}

\subsection*{Acknowledgements} I thank Emilio Corso, Eino Rossi and Tuomas Sahlsten for useful discussions. I am also grateful to two anonymous referees for their careful reading of the manuscript that helped considerably improve the presentation.

\section{Uniformization lemmas}

In this section we collect some  ``uniformization lemmas'' that will be used in the proof of Theorem \ref{thm:inverse-thm-A+A}. The following basic result appears implicitly in several papers of J.~Bourgain and was explicitly recorded in \cite[Lemma 3.4]{KeletiShmerkin19}.
\begin{lemma}\label{lem:uniform-subset}
  Let $L,S\in \N$, and let $A$ be a $2^{-m}$-set in $[0,1)^d$, where $m=S L$. Then there exists an $(L,S)$-uniform subset $A'\subset A$  such that
  \[
    |A'|\ge (2L d)^{-S}|A| = 2^{(-\log(2Ld)/L) m}|A|.
  \]
\end{lemma}

Given a set $A\subset [0,1)^d$ and a dyadic cube $I\in\cD_k$, we define the renormalization $A^I$ as
\[
  A^I = \mathbf{H}_I(A\cap I),
\]
where $\mathbf{H}_I$ is the orientation-preserving homothety mapping $I$ to $[0,1)^d$. If $\ell \le k$, we also let $A^I_\ell$ be the $2^{-\ell}$-set corresponding to $A^I$:
\[
  A^I_\ell = \big\{ j 2^{-\ell}: 2^{-\ell}[j+[0,1)^d) \cap A^I\neq\emptyset \big\}.
\]
The same idea behind Lemma \ref{lem:uniform-subset} yields the following more general result:
\begin{lemma} \label{lem:uniform-subset-general}
  Let $(F_s)_{s\in [S]}$ be functions defined on $2^{-L}$-sets and taking $\le V$ values. Then for any $2^{-SL}$-set $A_0$, there exists a $(L,S)$-uniform subset $A$ such that:
  \begin{itemize}
    \item $|A|\ge  (2L Vd)^{-S}|A_0|$,
    \item For each $s\in [S]$, the function $F_s$ is constant on $A^{I}_{L}$ over all $I\in\cD_{sL}(A)$.
  \end{itemize}
\end{lemma}
\begin{proof}
  The idea is to prune the $A_0$-tree backwards, ensuring the desired uniformity and constancy at each scale.

  To begin, we pigeonhole a value $k\in [dL]$ such that
  \[
    \sum\big\{|A_0\cap I| : I\in\cD_{(S-1)L}, |A_0\cap I|_{L} \in (2^{k-1},2^k] \} \ge (dL)^{-1}|B|.
  \]
  For each $I$ with $|A_0\cap I|_{L} \in (2^{k-1},2^k]$, let $B^I_L$ be a subset of $(A_0)_L^I$ of size $2^{k-1}$ (so, the size drops by at most a factor of $2$), and let $\widetilde{B}_{S-1}$ be the union of all such $B^I_L$. Then,
  \[
    |\widetilde{B}_{S-1}| \ge (2dL)^{-1}|A_0|.
  \]
  Next, we pigeonhole a value $v_{S-1}$ in the range of $F_{S-1}$ such that
  \[
    \sum\left\{|B\cap I| : I\in\cD_{(S-1)L}, F_{S-1}\bigl((\widetilde{B}_{S-1})^I_L\bigr)=v_{S-1}\right\} \ge V^{-1}|\widetilde{B}_{S-1}| .
  \]
  Let $B_{S-1}$ be the union of the $\widetilde{B}_{S-1}\cap I$ appearing in the left-hand side. Thus,
  \[
    |B_{S-1}|\ge V^{-1}|\widetilde{B}_{S-1}| \ge (2dVL)^{-1}|A_0|.
  \]

  We replace $A_0$ by $B_{S-1}$ and continue by backwards induction in $s$. Eventually we reach a set $B_0$, and define $A:=B_0$. Clearly, $|A|\ge (2L Vd)^{-S}|A_0|$.

  By construction, if $I\in \cD_{sL}(A)$ for $0\le s\le S-1$, then $I\in \cD_{sL}(B_s)$, and $A_L^I = (B_s)_L^I$. Thus, the constancy of $F$ and the branching number at each level is preserved throughout the process, and holds for $A$ as well.
\end{proof}

Given a set $A\subset[0,1)^d$ and $L\in\N$, let $D_L(A)$ be the smallest integer $j$ such that $A\subset W\bigl((\sqrt{d}+1)2^{-L}\bigr)$ for some $W\in\mathbb{A}(d,j)$. Since $j=d$ always works, this is well-defined and takes values in $[d+1]$. Applying Lemma \ref{lem:uniform-subset-general} to $F=D_L$ and $V=d+1$, we get:
\begin{lemma} \label{lem:uniform-subset-subspace}
  Let $L,S\in \N$, and let $A_0$ be a $2^{-SL}$-set in $[0,1)^d$. Then there exists a $(L,S)$-uniform subset $A\subset A_0$  such that
  \[
    |A|\ge (2d (d+1) L)^{-S}|A_0|,
  \]
  and, for each $s\in[S]$, the dimension $D_L(A^I_L)$ is constant over all $I\in\cD_{sL}(A)$.
\end{lemma}

The same idea yields the following lemma, asserting that one can always find a large subset for which all points are near the center of dyadic cubes, after translation.
\begin{lemma} \label{lem:center-by-translation}
  Let $A\subset[1/3,2/3)^d$ be a $2^{-SL}$-set. There are $A'\subset A$ with $|A'|\ge 3^{-2d S} |A|$ and a vector $y\in[-1/3,1/3)^d$ such that  $x\in \frac{1}{3}\mathcal{D}_{sL}(x)$ for all $x\in A'+y$ and $s\in [S]$.
\end{lemma}
\begin{proof}
  Given a cube $I$ in $[0,1)^d$,  let $\{ I^{(j)}\}_{j\in\{-1,0,1\}^d}$ be the partition of $I$ into $3^d$ non-overlapping cubes of side length a third that of $I$, indexed in the natural way. For example, $I^{(0,\ldots,0)}=\tfrac{1}{3}I$.

  Let $A_S = A$. Once $A_{s+1}$ has been defined for $s\in [S]$, we define $A_{s}$ as follows: for each $I\in \cD_{sL}(A_{s+1})$, pigeonhole an index $j(I)\in \{-1,0,1\}^{d}$ such that
  \[
    \big|A_{s+1} \cap I^{(j(I))}|\ge 3^{-d}|A_{s+1}\cap I|.
  \]
  Then select $j(s)$ such that
  \[
    \sum \{ |A_{s+1}\cap I|: j(I)=j(s) \} \ge 3^{-d}|A_{s+1}|.
  \]
  Set
  \[
    A_s = \bigcup\big\{ A_{s+1} \cap I^{(j(s))} : j(I)=j(s) \big\}.
  \]
  Then $|A_s| \ge 3^{-2d}|A_{s+1}|$.  The claim holds with $A'=A_1$ and $y= - \sum_{s=1}^{S-1} 3^{-(s+1)} j(s)$.
\end{proof}

We conclude this section by recalling another simple but useful result, stating that one can always ``collapse the branching'' of a uniform set at a subset of scales.
\begin{lemma} \label{lem:collapse-branching}
  Given $L,S\in\N$ the following holds. Let $A_0$ be $(L,S,R_s)$-uniform. Let $\mathcal{S}\subset [S]$ be any set, and for each $s\in\mathcal{S}$ and $I\in \cD_{sL}(A_0)$ let $X^I\subset (A_0)^I_L$ be a subset with $|X^I|=R'_s$ (independent of $I$).

  Then, there exists a $(L,S,R''_s)$-uniform set $A\subset A_0$, such that:
  \begin{enumerate}[(a)]
    \item If $s\in\mathcal{S}$, then $A^I_L = X^I$ for all $I\in\cD_{sL}(A)$. In particular, $R''_s=R'_s$.
    \item If $s\notin\mathcal{S}$, then $A^I_L= (A_0)^I_L$ for all $I\in\cD_{sL}(A)$. In particular, $R''_s=R_s$.
    \item
          \[
            |A| \ge \left(\prod_{s\in\mathcal{S}} \frac{R'_s}{R_s}\right)\cdot |A_0|.
          \]
  \end{enumerate}
\end{lemma}
\begin{proof}
  The case where $X^I$ is a singleton (equivalently, $R'_s=1$) for $s\in\mathcal{S}$ is \cite[Lemma 3.7]{Shmerkin19}. The general case follows in the same way, pruning the tree backwards, replacing $A^I_L$ by $X^I$ at each step when $s\in\mathcal{S}$.
\end{proof}

\section{Proof of  Theorem \ref{thm:inverse-thm-A+A}}

\subsection{Hochman's theorem on saturation of self-convolutions}

The proof of Theorem \ref{thm:inverse-thm-A+A} is based on M.~Hochman's work on inverse theorems for the entropy of convolutions \cite{Hochman17}, and specifically \cite[Theorem 4.2]{Hochman17}. We start by reviewing the concepts required to state this result.

We write $\cP_d$ for the Borel probability measures on $[0,1)^d$. Given $\mu\in\cP_d$ and $x,k$ such that $\mu(\cD_k(x))>0$, we denote
\[
  \mu^{x,k} = \frac{1}{\mu(\cD_k(x))} \mathbf{H}_{x,k}\left(\mu|_{\cD_k(x)}\right),
\]
where $\mathbf{H}_{x,k}$ is the (orientation-preserving) homothety mapping $\cD_{x,k}$ to $[0,1)^d$.

Given a finite set $I\subset\N$ and a family of measures $\mathcal{M}\subset\mathcal{P}([0,1)^d)$, we use the notation
\[
  \mathbb{P}_{i\in I}(\mu^{x,i}\in \mathcal{M}) = \frac{1}{|I|} \sum_{i\in I}  \mu\{ x: \mu^{x,i}\in\mathcal{M} \}.
\]
Likewise, if $F:\mathcal{P}([0,1)^d)\to \R$ is a Borel function, we denote
\[
  \mathbb{E}_{i\in I}(F(\mu^{x,i})) = \frac{1}{|I|} \sum_{i\in I}  \int  F(\mu^{x,i})\,d\mu(x).
\]

Given a measure $\mu\in\mathcal{P}([0,1)^d)$ and $m\in\N$, we denote by
\[
  H_m(\mu) = \frac{1}{m} H(\mu,\cD_m)
\]
the \emph{normalized} entropy of $\mu$ on the dyadic partition $\cD_m$.

The following simple but fundamental lemma relates global and local entropies; see \cite[Section 3.2]{Hochman17}.
\begin{lemma}[Hochman] \label{lem:local-to-global-entropy}
  Let $\mu\in\cP([0,1)^d)$. Then
  \[
    H_{SL}(\mu) = \mathbb{E}_{i\in L[S]} H_L(\mu^{x,i}).
  \]
\end{lemma}

\begin{definition} \label{def:concentration-saturation}
  Fix a measure $\mu\in\cP(\R^d)$, a linear subspace $V\leq \R^d$, and $\e>0$, $L\in\N$.

  We say that $\mu$ is \textbf{$(V,\e)$-concentrated} if there is $x$ such that
  \[
    \mu(x+V^{(\e)}) \ge 1-\e,
  \]
  where $X^{(\e)}$ denotes the open $\e$-neighborhood of $X$.

  We say that $\mu$ is \textbf{$(V,L)$-saturated} if
  \[
    H_L(\mu) \ge H_L(\Pi_{V^\perp}\mu) + \dim(V)-\frac{C}{L},
  \]
  where $C$ is a suitable constant depending only on the ambient dimension $d$. \footnote{In \cite{Hochman17}, the definition of saturation has an additional parameter $\e>0$; however, in order for the proof of \cite[Theorem 4.2]{Hochman17} to work, one has to take $\e=O(1/L)$, which matches our definition.}

\end{definition}

The following corollary of \cite[Theorem 4.12]{Hochman17} is our main tool in the proof of the inverse theorem:
\begin{theorem}  \label{thm:hochman-conv-saturated}
  Given $\eta>0$ and $L\in\N$, there is a number $k=k(\eta,L)>0$ such that the following holds for large enough $S$: let $\mu\in\cP([0,1)^d)$ and let $\tau=\mu^{*k}$ be its $k$-th convolution power. Then there are subspaces $(V_{s L})_{s\in[S]}$ of $\R^d$ such that
  \begin{align*}
    \mathbb{P}_{L[S]} (\tau^{x,i} \text{ is $(V_i,L)$-saturated })      & \ge 1-\eta, \\
    \mathbb{P}_{L[S]} (\mu^{x,i} \text{ is $(V_i,\eta)$-concentrated }) & \ge 1-\eta
  \end{align*}
\end{theorem}

This is the same statement of  \cite[Theorem 4.12]{Hochman17}, except that we have $[S]L$ instead of $[SL+1]$. The version above follows from the observation that if $I\subset J$ then
\[
  \mathbb{P}_I(\nu^{x,i}\in\mathcal{M}) \le \frac{|J|}{|I|} \mathbb{P}_J(\nu^{x,i}\in\mathcal{M}).
\]

\subsection{Proof of Theorem \ref{thm:inverse-thm-A+A}}

We use standard $O(\cdot)$ notation to hide unimportant constants.  All the implicit constants below may depend on the ambient dimension $d$. Let $L$ be large enough to be fixed later. Let $k=k(L)$ be the number given by Theorem \ref{thm:hochman-conv-saturated} for
\begin{equation} \label{eq:def-eta}
  \eta:=2^{-(d+1)L}
\end{equation}
Suppose $|A+A|\le 2^{\sigma m}|A|$. By the Pl\"{u}nnecke-Ruzsa inequalities (see e.g. \cite[Corollary 6.28]{TaoVu06}), this implies
\begin{equation} \label{eq:PR}
  |kA| \le 2^{\sigma k m}|A|.
\end{equation}
Applying Lemma \ref{lem:uniform-subset-subspace} and recalling that $m=SL$, let $A_0$ be an $L$-uniform subset of $A$ such that
\begin{equation} \label{eq:A_0-large-subset}
  |A_0|\ge (2d (d+1) L)^{-S}|A|,
\end{equation}
and there are numbers $k_s\in [d+1]$ such that for all $I\in\cD_{sL}(A_0)$ there is a minimal affine subspace $W_I$ of dimension $k_s$ such that
\begin{equation} \label{eq:A_0-concentration}
  A_0\cap I \subset W_I\big((\sqrt{d}+1) 2^{-(s+1)L}\big).
\end{equation}
Recalling that $m=LS$, we get that
\[
  \log|k A_0| \overset{\eqref{eq:PR}}{\le}  \sigma k m + \log|A| \overset{\eqref{eq:A_0-large-subset}}{\le} \left(\sigma k+ \frac{O_d(1)+\log L}{L}\right)m + \log |A_0|.
\]
We pick $\sigma=\sigma(L)=(kL)^{-1}\log L$ (recall that $k=k(L)$). Hence, we have
\begin{equation} \label{eq:sumset-small-1}
  |k A_0| \le 2^{O(\log L/L)m} |A_0|.
\end{equation}

Let $\nu$ be the uniform probability measure on $A_0$. Set also $\tau=\nu^{*k}$. By Theorem \ref{thm:hochman-conv-saturated}, there exist subspaces $V_{0},V_{L},\ldots, V_{(S-1)L} \leq \R^d$ such that
\begin{align}
  \mathbb{P}_{i\in L[S]} (\tau^{x,i} \text{ is $(V_i,L)$-saturated })      & > 1-  \eta,  \label{eq:saturated}   \\
  \mathbb{P}_{i\in L[S]} (\nu^{x,i} \text{ is $(V_i,\eta)$-concentrated }) & > 1-  \eta. \label{eq:concentrated}
\end{align}

Let
\[
  \mathcal{S} = \{ s\in [S]: \nu^{x,sL} \text{ is $(V_{sL},\eta)$-concentrated for some $x$} \}.
\]
It follows from \eqref{eq:concentrated} that
\begin{equation} \label{eq:S-large}
  |[S]\setminus \mathcal{S}| \le \eta S.
\end{equation}

Since the set $A_0$ is $(L,S)$-uniform, the measures $\nu^{x,sL}$ are all uniform (they give the same mass to all points in their support). We claim that if $\nu^{x,sL}$ is $(V_{sL},\eta)$-concentrated, then $\nu^{x,sL}$ must be supported on a $(\sqrt{d}+1)2^{-(s+1)L}$-neighborhood of a translate of $V_{sL}$. Indeed, otherwise there is a cube  $J\in\cD_{L}$ with $\nu^{x,sL}(J)>0$ that does not meet the $2^{-(d+1)L}$ neighborhood of the translate of $V_{sL}$ given by the definition of concentration. But then
\[
  \nu^{x,sL}(J) \ge 2^{-dL} \overset{\eqref{eq:def-eta}}{>} \eta,
\]
as all cubes in $\cD_L(\supp(\nu^{x,sL}))$ have equal mass and there are at most $2^{dL}$ of them. This yields a contradiction with the
definition of concentration.

In light of the claim and the minimality of $W_I$ in \eqref{eq:A_0-concentration}, we deduce that
\begin{equation} \label{eq:comp-dim-subspaces}
  k_s \le \dim(V_{sL}) \quad\text{for all } s\in\mathcal{S}.
\end{equation}

It follows directly from Definition \ref{def:concentration-saturation} that if $\rho$ is $(V,L)$-saturated then $H_L(\rho) \ge \dim(V)-O(1/L)$. Applying  Lemma \ref{lem:local-to-global-entropy} and \eqref{eq:saturated}, and recalling that $\eta=2^{-(d+1)L}\le 1/L$, we get that
\[
  H_m(\tau) \ge   (1-\eta)\left(\frac{1}{S}\sum_{s\in[S]}\dim(V_{s L})-O(1/L)\right)  \ge -O(1/L)+ \frac{1}{S}\sum_{s\in [S]}  \dim(V_{s L}).
\]
Since $\tau$ is supported on $kA_0$,
\[
  \log|k A_0| \ge H(\tau,\cD_m) \ge  -O(1/L) m + L\sum_{s\in [S]}  \dim(V_{s L}).
\]
Recalling \eqref{eq:sumset-small-1}, we deduce that
\[
  \log|A_0| \ge -O(\log L/L)m + L\sum_{s\in \mathcal{S}}  \dim(V_{s L}).
\]
Writing $(R_s)_{s\in [S]}$  for the branching numbers of $A_0$, we also have
\[
  \log|A_0| = \prod_{s\in [S]} \log R_s \le |[S]\setminus\mathcal{S}|d L + \sum_{s\in \mathcal{S}} \log R_s.
\]
Since
\[
  |[S]\setminus\mathcal{S}|d L  \overset{\eqref{eq:S-large}}{\le} \eta d S L \overset{\eqref{eq:def-eta}}{<} (\log L/L)m,
\]
we see that
\[
  \sum_{s\in \mathcal{S}} \log R_s \ge -O(\log L/L)m+ \sum_{s\in\mathcal{S}} L \dim(V_{s L}).
\]
On the other hand, we get from \eqref{eq:A_0-concentration} and \eqref{eq:comp-dim-subspaces} that
\[
  \log R_s \le O(1) + L \dim(V_{s L}), \quad s\in\mathcal{S}.
\]
Thus, by Markov's inequality,
\[
  |\{s\in\mathcal{S}:  \log R_s < L(\dim V_{s L}-L^{-1/2}) \}| \le O(L^{-1/2}\log L)S.
\]
In light of \eqref{eq:S-large}, the above also holds if we replace $\mathcal{S}$ by $[S]$ on the left-hand side. The desired set $A'$ is obtained from $A_0$ by using Lemma \ref{lem:collapse-branching} to ``collapse the branching'' of $A_0$ to $1$ at the scales $s$ for which $\log R_s < L(\dim V_{s L}-L^{-1/2})$. We redefine $R_s$ to be $1$ and $k_s$ to be $0$ for these collapsed scales.

To conclude, take $L$ large enough that $C L^{-1/2}\log L\le \delta$ for a large constant $C$ (possibly depending on $d$). Property \eqref{it:inverse-A+A-i} holds by construction. To verify \eqref{it:inverse-A+A-ii}, note that (using \eqref{eq:S-large} once again), we have
\[
  |A'| \ge 2^{O(L^{-1/2}\log L) m} |A_0| \ge  2^{O(L^{-1/2}\log L) m}  |A| \ge 2^{-\delta m}|A|.
\]
Finally, \eqref{it:inverse-A+A-iii} follows from \eqref{eq:comp-dim-subspaces} and the fact that, by construction, either $k_s=0$ or $\log R_s \ge L(\dim V_{s L}-L^{-1/2})$.

\subsection{Proof of Theorem \ref{thm:inverse}}

Equipped with Theorem \ref{thm:inverse-thm-A+A}, the proof of Theorem \ref{thm:inverse} follows roughly the same steps as the proof of the inverse theorem from \cite[Theorem 2.1]{Shmerkin19}. We apply several lemmas from \cite{Shmerkin19}; some were stated in \cite{Shmerkin19} only in dimension $1$ but the same proof works in higher dimensions. We keep allowing all implicit constants to depend on the ambient dimension $d$.

Let $\tau>0$ be a small parameter to be chosen shortly in terms of $L$, and define $\e$ by
\[
  \e = \frac{\tau}{2\max(q,q')},
\]
where $q'$ is the dual exponent of $q$ (i.e. $q=q/(q-1)$). Assume that $\mu,\nu$ are as in the statement of Theorem \ref{thm:inverse}, and
\[
  \|\mu*\nu\|_q \ge 2^{-\e m}\|\mu\|_q.
\]
Applying \cite[Lemmas 3.3 and 3.4]{Shmerkin19}, we get $2^{-m}$-sets $A_1,B_1$ such that:
\begin{enumerate}[\rm(a)]
  \item \label{it:a} $\mu,\nu$ are constant up to a factor of $2$ on $A_1,B_1$ respectively.
  \item \label{it:b} $\| \1_{A_1}*\1_{B_1}\|_2^2 \ge 2^{-\tau m} |A_1||B_1|^2$, where $\1_X$ denotes the indicator function of a set $X$, and $\|\cdot\|_2$ is the $L^2$ norm of a finitely supported measure as defined in \eqref{eq:def-q-norm}.
  \item \label{it:c} $\|\mu|_{A_1}\|_q \ge 2^{-2\e m}\|\mu\|_q$ \text{ and } $\|\nu|_{B_1}\|_1 \ge 2^{-2\e m}$.
\end{enumerate}
Let $L$ be large enough in terms of $\delta$ that Theorem \ref{thm:inverse-thm-A+A} applies, and let $\sigma=\sigma(L)$ be the number provided by Theorem \ref{thm:inverse-thm-A+A}. Let $\tau=\tau(\sigma)$ be the number provided by the Asymmetric Balog-Szemer\'{e}di-Gowers Theorem in the form given in \cite[Theorem 3.2]{Shmerkin19} (with $\sigma$ in place of $\kappa$). Even though stated in \cite{Shmerkin19} in the real line, this theorem holds in any Abelian group. The inequality \eqref{it:b} is precisely the assumption of the Asymmetric Balog-Szemer\'{e}di-Gowers Theorem. Note that the parameters follow the chain of dependencies $\delta\to L\to\sigma\to\tau\to\e$ (and $\e$ depends additionally on $q$); later we will impose additional constraints of some on the parameters that respect these dependencies.

Applying the asymmetric Balog-Szemer\'{e}di-Gowers Theorem, we get $2^{-m}$-sets $H_0,X$ such that:
\begin{enumerate}[\rm(i)]
  \item \label{it:i} $|H_0+H_0|\le 2^{\sigma m}|H_0|$.
  \item \label{it:ii} $|A_1\cap (X+H_0)| \ge 2^{-\sigma m}|A_1| \ge 2^{-2\sigma m}|X||H_0|$.
  \item \label{it:iii} $|B_1\cap H_0|\ge 2^{-\sigma m}|B_1|$.
\end{enumerate}

Let $H$ be the set obtained by applying Theorem \ref{thm:inverse-thm-A+A} to $H_0$, with branching numbers $(R_s)_{s\in [S]}$, and dimension sequence $(k_s)_{s\in [S]}$. Our choice of parameters ensures that Theorem \ref{thm:inverse-thm-A+A} is indeed applicable.

At this stage we apply Lemma \ref{lem:center-by-translation} to $A_1\cap (X+H_0)$ and $B_1\cap H_0$; since the sets $A, B$ will ultimately be subsets of these sets, this ensures that claim \ref{C} holds. This refinement only decreases the size of $A_1\cap (X+H_0)$ and $B_1\cap H_0$ by a factor that can be absorbed into $2^{-\sigma m}$ (at the cost of halving the value of $\sigma$), so for notational simplicity we keep denoting these refinements by $A_1\cap (X+H_0)$ and $B_1\cap H_0$.

Now comes the most significant change with respect to \cite{Shmerkin19}. Instead of splitting the scales $s$ according to a full branching/no branching dichotomy, we need to consider all the possible dimensions $k_s$ coming from Theorem \ref{thm:inverse-thm-A+A}. Given $s\in [S]$ and a $2^{-L}$-set $Y$, we let
\[
  F_s(Y) = \left\lfloor   \big(\inf\left\{ \log|\pi_V Y|_{2^{-L}}: V\in\mathbb{G}(d,d-k_s) \right\} \big) \right\rfloor \in [0,L d].
\]
Unpacking the definition, this means that
\[
  |\pi_V Y|_{2^{-L}} \ge 2^{F_s(Y)} \quad \text{for all } V\in\mathbb{G}(d,d-k_s),
\]
and there is $V\in \mathbb{G}(d,d-k_s)$ such that $|\pi_V Y|_{2^{-L}} \le 2^{F_s(Y)+1}$.

Applying Lemma \ref{lem:uniform-subset-general} to $A_1\cap (X+H_0)$, we get a uniform subset $A\subset A_1\cap (X+H_0)$ such that $F_s$ is constant on each $A_L^I$, $I\in\cD_{sL}(A)$, and
\begin{equation} \label{eq:A-large-subset-A_1}
  |A| \ge 2^{-O(\log L/L)m}|A_1\cap (X+H_0)| \overset{\eqref{it:ii}}{\ge} 2^{-(O(\log L/L)+\sigma)m} |A_1|.
\end{equation}
We make $\sigma$ smaller if needed to ensure that $\sigma\le \log L/L$. Since $A+H\subset X+H_0+H_0$, it follows that
\begin{equation} \label{eq:A+H-upper-bound}
  |A+H| \le |X||H_0+H_0| \overset{\eqref{it:i}}{\le}  2^{\sigma m} |X||H_0| \overset{\eqref{it:ii}}{\le} 2^{2\sigma m} |A_1| \overset{\eqref{eq:A-large-subset-A_1}}{\le} 2^{O(\log L/L)m}|A|.
\end{equation}

Let $M_s$ be the constant value of $F_s$ at the sets $A_L^I$, $I\in\cD_{sL}(A)$.
\begin{lemma} \label{lem:A+H-lower-bound}
  There exists a constant $c\in (0,1)$ such that
  \begin{equation} \label{eq:A+H-lower-bound}
    |A+H| \ge c^S \prod_{s=0}^{S-1}  2^{(1-\delta) L k_s} \cdot 2^{M_s}.
  \end{equation}
\end{lemma}
\begin{proof}
  Fix $s\in [S]$, $I\in\cD_{sL}(A)$, $J\in\cD_{sL}(H)$. By Theorem \ref{thm:inverse-thm-A+A}, $R_s\ge 2^{(1-\delta)L k_s}$ and there is  $W_J\in\mathbb{A}(d,k_s)$ such that $H^J\subset W_J(C_d 2^{-L})$. Since the projection of $A^I$ to $W_J^\perp$ has $2^{-L}$-covering number $\ge 2^{M_s}$, it follows that
  \[
    | (A\cap I) + (H\cap J)|_{2^{-(s+1)L}}  = |A^I+H^J|_{2^{-L}} \gtrsim R_s 2^{M_s} \ge 2^{(1-\delta) L k_s} 2^{M_s}.
  \]
  We emphasize that this bound is independent of $I,J$ (for a given $s$). It follows that, given $I\in\cD_{sL}(A)$, $J\in\cD_{sL}(H)$, there are $\gtrsim 2^{(1-\delta) L k_s} 2^{M_s}$ pairwise disjoint sets of the form $I'+J'$, where $I'\in\cD_{(s+1)L}(A\cap I)$, $J'\in\cD_{(s+1)L}(H\cap J)$.

  Applying this with $s=0$, we obtain a collection $\mathcal{K}_0\subset\cD_{L}(A)\times \cD_{L}(H)$ such that:
  \begin{enumerate}[(a)]
    \item  $\{ I+J: (I,J)\in\mathcal{K}_0\}$ are pairwise disjoint,
    \item $|\mathcal{K}_0|\ge c \cdot 2^{(1-\delta) L k_0} 2^{M_0}$.
  \end{enumerate}
  For each $(I,J)\in\mathcal{K}_0$, we apply the same argument to obtain
  a family
  \[
    \mathcal{K}_{(I,J)}\subset \cD_{2L}(A\cap I)\times \cD_{2L}(H\cap J),
  \]
  such that the analogous properties hold. Let $\mathcal{K}_1 = \bigcup_{(I,J)\in\mathcal{K}_0} \mathcal{K}_{(I,J)}$. Then $\mathcal{K}_1\subset \cD_{2L}(A)\times \cD_{2L}(H)$, and
  \begin{enumerate}[(a)]
    \item $\{ I+J: (I,J)\in\mathcal{K}_1\}$ are pairwise disjoint,
    \item $|\mathcal{K}_1|\ge c 2^{(1-\delta) L k_0} 2^{M_0} \cdot c 2^{(1-\delta) L k_1}  2^{M_1}$.
  \end{enumerate}
  The claim follows by iterating this argument.
\end{proof}

Let $(R'_s)_{s\in [S]}$ be the branching numbers of $A$. By definition of $M_s$, for any $I\in\cD_{sL}(A)$, there is an orthogonal projection $\pi$ onto a $(d-k_s)$-dimensional plane such that $|\pi(A^I)|_{2^{-L}}\le 2^{M_s+1}$, so we have
\begin{equation} \label{eq:saturated-branching-number}
  \log R'_s \le  L k_s +  M_s +O(1).
\end{equation}
Combining \eqref{eq:A+H-upper-bound} and \eqref{eq:A+H-lower-bound}, and taking $L$ large enough in terms of $\delta$ that $\log L/L\le \delta$, we get
\[
  \sum_{s\in [S]} L k_s+ M_s - \log R'_s + O(1) \le O(\delta)m.
\]
Recalling \eqref{eq:saturated-branching-number} and applying Markov's inequality, we deduce that
\[
  \log R'_s \ge (1-\delta^{1/2})(L k_s+ M_s)
\]
for all $s$ outside an exceptional set $\mathcal{E}$ with
\[
  \sum_{s\in\mathcal{E}} \log R'_s \le \delta^{1/2}\log|A|.
\]
Collapsing the branching for scales in $\mathcal{E}$ to $1$ via Lemma \ref{lem:collapse-branching}, we get all the desired properties for $A$ (with $\delta^{1/2}$ in place of $\delta$).

To construct $B$ (from $B_1$) we proceed in a similar fashion. We uniformize $B_1\cap H_0$ so that $F_s$ (same function as above) is constant on each scale $s$. Then, since $B\subset B_1\cap H_0\subset H_0$,
\[
  |B+H| \le |H_0+H_0| \le 2^{\delta m}|H_0| \le 2^{(\delta+\sigma)m} |H| \le 2^{(\delta+\sigma)m} \prod_{s=0}^{S-1} 2^{L k_s+O(1)}.
\]
Let us denote again by $M_s$ the value of $F$ at scale $s$, now for the set $B$. Then Lemma \ref{lem:A+H-lower-bound}, this time applied to $B$, yields that
\[
  |B+H| \gtrsim  \prod_{s=0}^{S-1}  2^{(1-\delta) L k_s}\cdot 2^{M_s}.
\]
Thus, comparing upper and lower bounds for $|B+H|$, and taking $\sigma\le\delta$ as we may,
\begin{equation} \label{eq:branching-loss-B}
  \prod_{s=0}^{S-1}  2^{M_s} \le O(\delta) m.
\end{equation}
Note that, by definition of $M_s$, for each $I\in\cD_{sL}(B)$ there is a $(d-k_s)$-dimensional subspace $W_I$ such that $|\pi_{W_I} B_L^I|_{2^{-L}} \le 2^{M_s+1}$. By pigeonholing, there is an affine plane $V_I$ of dimension $k_s$ such that
\[
  \left|B_L^I\cap V_I\bigl((\sqrt{d}+1)2^{-L}\bigr)\right| \gtrsim 2^{-M_s} |B_L^I|.
\]
Thus, we can apply Lemma \ref{lem:collapse-branching}, replacing each $B_L^I$ by a suitable subset, so that \ref{B4} holds for this choice of $V_I$ (and uniformity is preserved). Doing so reduces the cardinality of each $B_L^I$ by a factor $\lesssim 2^{M_s}$, and so by \eqref{eq:branching-loss-B} the total loss in cardinality is an acceptable factor $2^{O(\delta)m}$. This finishes the proof of Theorem \ref{thm:inverse}.

\section{Connections and applications}

\subsection{\texorpdfstring{$L^q$ dimensions}{Dimensions} of dynamical self-similar measures, and applications}

In \cite{Shmerkin19}, the author applied the one-dimensional version of Theorem \ref{thm:inverse} to compute the $L^q$ dimensions of a class of dynamically driven self-similar measures under a very mild exponential separation assumption. In a joint work with E.~Corso \cite{CorsoShmerkin24}, we make crucial use of Theorem \ref{thm:inverse} to extend this result to arbitrary dimensions, under suitable assumptions. Just as in the one-dimensional case, this has a broad range of applications to the dimension theory of self-similar sets and measures, their projections and their slices.  We refer the reader to \cite{CorsoShmerkin24} for details, and to \cite[\S 3.8]{Shmerkin23} for an announcement of some of these applications.

\subsection{Khalil's inverse theorem}

Recently, O.~Khalil \cite[Proposition 11.10]{Khalil23} proved a related inverse theorem for the $L^q$ norms of convolutions, motivated by a striking application to exponential mixing of geodesic flows. Roughly speaking, Khalil considers the case in which $\nu$ is a $2^{-m}$-measure such that
\begin{equation} \label{eq:hyperplane-decay}
  \nu\bigl(H(\eta r)\cap B(x,r)\bigr) \le C\, \eta^{\beta}\nu(B(x,r))
\end{equation}
for some $C,\beta>0$, all $\eta>0$, all hyperplanes $H$ in $[0,1)^d$, and ``almost all'' points $x$ and scales $r\in [2^{-m},1]$ (recall that $X(r)$ stands for the $r$-neighborhood of a set $X$). He shows that under this assumption, if $\mu$ is any $2^{-m}$-measure with $\|\mu\|_q\le 2^{-\e q m}$ for some $\e>0$, then
\begin{equation}
  \label{eq:flattening-Khalil}
  \|\mu*\nu\|_q \le 2^{-\sigma m}\|\mu\|_q,
\end{equation}
where $\sigma=\sigma(q,\beta,\e)>0$. The nature of the ``almost all'' points and scales condition in Khalil's theorem has no direct correlate in Theorem \ref{thm:inverse}. Nevertheless, at least at a moral level Theorem \ref{thm:inverse} is more general: the assumption \eqref{eq:hyperplane-decay} ensures that (in the context of Theorem   \ref{thm:inverse}), if \eqref{eq:flattening-Khalil} fails, then $k_s=d$ for ``almost all'' scales $s$, and this can be seen to imply that $\|\mu\|_q > 2^{-\e q m}$.

Khalil's proof also relies on the results of M.~Hochman \cite{Hochman17}, but he applies a different result of Hochman, namely \cite[Theorem 2.8]{Hochman17}. As pointed out in the introduction, this is a result that can be seen as an entropy version of Theorem \ref{thm:inverse}.

We remark that the application of Theorem \ref{thm:inverse} to the $L^q$ dimensions of dynamical self-similar measures in \cite{CorsoShmerkin24} requires the full strength of Theorem \ref{thm:inverse}, and cannot be deduced from Khalil's result. Indeed, in the application of  Theorem \ref{thm:inverse} in \cite{CorsoShmerkin24}, the only information we have about $\nu$ is that $\|\nu\|_{q} \le 2^{-\rho m}$ for some small parameter $\rho>0$; in particular, nothing prevents $\nu$ from being supported on a single line, let alone a hyperplane.

\subsection{Additive energy of discretized sets}
\label{subsec:energy}

Given a finite set $X\subset \R^d$, we define its \emph{additive energy} by
\[
  \cE(X,X) = \# \{ (x_1,x_2,y_1,y_2)\in X^4: x_1+y_1=x_2+y_2 \}\;.
\]
This is a key concept in additive combinatorics, with applications throughout mathematics; see \cite[Section 2]{TaoVu06} for an introduction.

When $X$ is a $2^{-m}$-set, the additive energy can be expressed as the $L^2$ norm of convolutions. Indeed, if $\mu$ is the counting measure on $X$ (\emph{not} normalized), then
\[
  \cE(X,X) = \|\mu*\mu\|_2^2 \;.
\]
A trivial bound for $\cE(X,X)$ is given by $|X|^3$ (given $x_1,y_1,x_2$, the parameter $y_2$ is uniquely determined). In many applications one would like to know that there is a gain over this trivial bound. The results of this article provide mild conditions under which an exponential gain over $|X|^3$ is achieved. While it is possible to deduce this from Theorem \ref{thm:inverse}, the statement follows directly by combining Theorem \ref{thm:inverse-thm-A+A} with the (standard) Balog-Szemerédi-Gowers Theorem.

\begin{corollary} \label{cor:additive-energy}
  For each $\delta>0$ the following holds if $L\ge L_0(\delta)\in\N$ and $0<\sigma<\sigma(L)$,  for all sufficiently large $S=S(\delta,L,\sigma)$:

  Let $m=SL$. Let $X$ be a  $2^{-m}$ set in $[0,1)^d$ such that
  \[
    \cE(X,X)\ge 2^{-\sigma m}|X|^3.
  \]
  Then there exists a subset $A\subset X$ such that the following holds:
  \begin{enumerate}[\rm(i)]
    \item \label{it:BSG-A+A-i} $A$ is $(L,S,(R_s))$-uniform for some sequence $(R_s)_{s\in [S]}$.
          \item\label{it:BSG-A+A-ii}  $|A|\ge 2^{-\delta m}|X|$.
    \item \label{it:BSG-A+A-iii} For each $s\in [S]$, there is $k_s\in [0,d]$  such that
          \[
            \log R_s \ge L(k_s-\delta),
          \]
          and for each $I\in\cD_{sL}(A')$ there is $W_I\in\mathbb{A}(d,k_s)$ such that
          \[
            A\cap I\subset  W_{I}\big(2^{-(s+1)L}\big).
          \]
  \end{enumerate}
\end{corollary}
\begin{proof}
  By the Balog-Szemerédi-Gowers Theorem \cite[Theorem 2.29]{TaoVu06}, the assumption on the additive energy implies that there is $Y\subset X$ with $|Y|\gtrsim 2^{-O(\sigma)m}|X|$ such that
  \[
    |Y+Y| \lesssim 2^{O(\sigma)m}|Y|.
  \]
  The claim follows by applying Theorem \ref{thm:inverse-thm-A+A} to $Y$.
\end{proof}
Roughly speaking, the above corollary says that if $X$ has nearly maximal additive energy in the exponential sense, then at almost all scales it looks locally like an affine plane (with the dimension of the plane depending only on the scale).

We say that a $2^{-m}$-set $X$ is Ahlfors-regular between scales $2^{-m}$ and $1$ if there exist constants $C\ge 1$, $t\in [0,d]$ such that, for each $x\in X$ and $r\in [2^{-m},1]$,
\[
  C^{-1}\,r^{t}\,|X| \le |X\cap B(X,r)|_{2^{-m}} \le C\,r^{t}\,|X|.
\]
In dimension $d=1$, S.~Dyatlov and J.~Zahl \cite{DyatlovZahl16} proved that Ahlfors-regular sets in the sense above with $t\in (0,1)$ admit an exponential additive energy improvement, with the exponent depending quantitatively on the constants $C,t$. In an unpublished manuscript, B.~Murphy \cite{Murphy19} obtained nearly sharp bounds for the exponential decay exponent in this context. L.~Cladek and T.~Tao \cite{CladekTao21} developed a different quantitative approach in $d=1$ and obtained the first results in higher dimension: they proved that Ahlfors-regular sets have a quantitative exponential additive energy gain if $t\notin\Z$ (this condition is necessary in general - a counterexample is a set that looks nearly like a $d$-dimensional subspace). It is not hard to recover a non-quantitative form of this result using Corollary \ref{cor:additive-energy}. As a further example, we deduce additive energy flattening for sets that are ``porous on $k$-planes''. If $B$ is a ball, we denote its radius by $r(B)$.
\begin{definition} \label{def:porous}
  Let $k\in\{1,\ldots,d\}$ and $\rho,\eta\in (0,1)$. We say that a set $X\subset\R^d$ is \emph{$(k,\rho)$-porous between scales $\eta$ and $1$} if for every $W\in\mathbb{A}(d,k)$, and every ball $B$ with radius $r(B)\in [\eta,1]$ centered in $W$, there is $y\in  B\cap W$ such that
  \[
    B(y,\rho \cdot r)\cap X \cap W = \emptyset.
  \]
  We say that $X$ is $k$-porous between scales $\eta$ and $1$ if it is $(k,\rho)$-porous between scales $\eta$ and $1$ for some $\rho>0$, and drop reference to the scales if $\eta=0$ or is understood from context.
\end{definition}
This definition is equivalent to $X\cap W$ being upper $\rho$-porous (in the classical sense) for all $k$-planes $W$. If $X\subset\R^{d-k+1}$ is $1$-porous (for example, we can take $X= C^{d-k+1}$, where $C$ is the middle-thirds Cantor set), then $[0,1]^{k-1} \times X$ is $k$-porous, but it is not $(k-1)$-porous.

\begin{prop} \label{prop:application-flattening}
  Given $k\in \{1,\ldots,d\}$, $\rho\in (0,1)$ and $\lambda>0$, there is $\sigma=\sigma(k,\rho,\lambda)>0$ such that the following holds for all sufficiently large $m\ge m_0(k,\rho,\lambda)$:
  Let $X\subset[0,1)^d$ be a $2^{-m}$ set which is $(k,\rho)$-porous between scales $2^{-m}$ and $1$, and such that
  \[
    |X| \ge  2^{(k-1+\lambda) m}.
  \]
  Then
  \begin{equation} \label{eq:add-energy-flattening}
    \cE(X,X) \le 2^{-\sigma m}|X|^3.
  \end{equation}
\end{prop}
\begin{proof}
  To begin, we claim that, for every $j\in [m]$, each $W\in\mathbb{A}(d,k)$, and every $I\in\cD_j(W)$,
  \begin{equation}
    \label{eq:porous-estimate}
    \frac{1}{L}\log \big|W\cap I\cap X(2^{-j+L})\big|_{2^{-(j+L)}} \le k - c_d\frac{\rho^k}{\log(1/\rho)} + \frac{O(d,\rho)}{L}.
  \end{equation}
  Let $Q_0$ be a cube in $W$ of side length $\ell_0=d^{1/2}2^{-j}$ such that $W\cap I\subset Q_0$. Let $p$ be the smallest integer such that $2^{-p}\le \tfrac{1}{4}d^{-1/2}\rho$. We consider the $2^{-p}$-adic grid $(\mathcal{Q}_i)_{i\ge 0}$ of sub-cubes of $Q_0$ (thus, $Q_0$ is the only element of $\mathcal{Q}_0$, and cubes in $\mathcal{Q}_i$ have side length $2^{-ip}\ell_0$).

  Fix $i$ such that $2^{-pi}\ell_0\rho/2 > 2^{-(j+L)}$  and $Q'\in\mathcal{Q}_i$. Then $Q'$ contains a ball $B_Q$ with $r(B_Q)=2^{-pi}\ell_0/2$. Applying the definition of porosity to $B_Q$, we deduce that there is a ball $B'_Q\subset B_Q\subset Q'$ with $r(B'_q)=\rho\cdot r(B_q)$ such that
  \[
    B'_Q\cap X\cap W=\emptyset.
  \]
  By the assumption on $i$, if $B''_Q$ is the ball concentric with $B'_Q$ and half the radius (thus, $r(B''_Q)=\rho 2^{-pi}\ell_0/4$), then
  \[
    B''_Q\cap X(2^{-(j+L)})\cap W=\emptyset.
  \]
  In turn, by our choice of $p$, the ball $B''_Q$ contains a cube $Q''\in\cQ_{i+1}$. In short, each cube in $\cQ_{i}$ contains a cube in $\cQ_{i+1}$ which is disjoint from
  $X(2^{-(j+L)})\cap W$. Iterating this, we deduce that $Q_0\cap X$ can be covered by $(2^{pk}-1)^i$ cubes in $\mathcal{Q}_i$. The largest allowed value $i_{\max}$ of $i$ satisfies $2^{-p i_{\max}}\ell_0\rho \sim_{p} 2^{-(j+L)}$; recalling the definitions of $\ell_0$ and $p$, we get that $i_{\max} \ge  L/p - O_{d,\rho}(1)$. In turn, each cube in $\mathcal{Q}_{i_{\max}}$ can be covered by $2^d$ standard dyadic cubes of side length $\sim_d 2^{-i_{\max} p}\ell_0$. We deduce that
  \[
    |W\cap I\cap X(2^{-(j+L)})|_{2^{-(j+L)}}  \le |Q_0\cap X|_{2^{-(j+L)}} \lesssim_{d,\rho} \big[(2^{pk}-1)^{1/p}\big]^{L}.
  \]
  From here a small calculation yields the claim.

  Now let
  \[
    \delta = \min\left\{\frac{c_d}{4}\frac{\rho^k}{\log(1/\rho)},\frac{\lambda}{2}\right\},
  \]
  where $c_d$ is the constant from \eqref{eq:porous-estimate}. Let $L_0=L_0(\delta)$ be the constant from Corollary \ref{cor:additive-energy} and take $L\ge L_0$ large enough that
  \[
    2d O(d,\rho) \le \delta L,
  \]
  where $O(d,\rho)$ is the implicit constant from \eqref{eq:porous-estimate}. We have arranged things so that, for each $I\in\cD_{sL}$ and each $k$-plane $W$,
  \begin{equation} \label{eq:porosity-application}
    \big|W\cap I\cap X(2^{-(s+1)L})\big|_{2^{-(s+1)L}} \le 2^{(k-2d\delta) L}.
  \end{equation}

  Assume for sake of contradiction that \eqref{eq:add-energy-flattening} does not hold, with $\sigma$ given by Corollary \ref{cor:additive-energy} applied with the chosen values of $\delta$ and $L$. Let $Y$ be the set provided by the corollary, with branching numbers $(R_s)_{s\in [S]}$ and dimension sequence $(k_s)_{s\in S}$.

  Note that
  \[
    L \sum_{s\in [S]} k_s \ge  \sum_{s\in S} \log R_s = \log |Y| \ge \log |X| -\delta m = (k-1+\lambda-\delta)m.
  \]
  Thus, since $\delta\le \lambda/2$,
  \[
    \sum_{s\in [S]} k_s  \ge (k-1+\lambda/2)S.
  \]
  In particular, there is some $k_s\ge k$.  By Corollary \ref{cor:additive-energy}\eqref{it:BSG-A+A-iii}, there are $s\in [S]$, $I\in\cD_s$ and a $k'$-plane $W'$, with $k'\ge k$, such that $Y\cap I\subset W'(2^{-(s+1)L})$. Since $|Y\cap I|\ge R_s \ge 2^{(1-\delta)L k'}$, after translating $W'$ if needed we get
  \[
    | W' \cap I \cap X^{2^(-(s+1)L)} |_{2^{-(s+1)L}} \ge 2^{(1-\delta)L k'}.
  \]
  Pigeonholing, there is a $k$-plane $W\subset W'$ such that
  \[
    | W \cap I \cap X^(2^{-(s+1)L}) |_{2^{-(s+1)L}} \gtrsim  2^{(1-\delta)L k}.
  \]
  This contradicts  \eqref{eq:porosity-application}. This contradiction finishes the proof.
\end{proof}

The assumption that $|X|\ge 2^{(k-1+\lambda)m}$ is required since a $(k-1)$-plane is $k$-porous and has essentially maximal additive energy. Proposition \ref{prop:application-flattening} can be improved in several ways. An inspection of the proof shows that one does not need to have porosity at all scales, but only at \emph{most} scales, with ``most'' depending on the value of $\lambda$. Under further assumptions on $X$, for example if $X$ is Ahlfors-regular between scales $2^{-m}$ and $1$, one only needs porosity at a positive proportion of scales. One can deduce these facts from Corollary \ref{cor:additive-energy} using locally entropy averages in the spirit of \cite{Shmerkin12} in order to establish analogs of \eqref{eq:porous-estimate}  at a subset of scales. We defer a detailed discussion to a future work.

We remark that Theorem \ref{thm:inverse} can be used to show flattening of additive energy in more general settings. Given two Borel probability measures $\mu,\nu$ on $\R^d$, we define their \emph{additive energy at scale $r>0$} by
\[
  \cE_r(\mu,\nu) = (\mu^2\times\nu^2)\{ (x_1,x_2,y_1,y_2): |x_1+y_1-x_2-y_2|\le r \}\;.
\]
Again this is, up to constants, the same as the $L^2$ norm of convolutions. We leave the proof of the following standard lemma to the reader. Given a Radon measure $\mu$ on $\R^d$, we denote by $\mu^{(m)}$ any reasonable $2^{-m}$-discretization of $\mu$; for concreteness, let
\[
  \mu^{(m)}(x) = \mu\big(x+2^{-m}[-1/2,1/2)^d\big).
\]
\begin{lemma}
  Let $\mu,\nu$ be Radon measures on $[0,1]^d$. Then, for every $m\in\N$,
  \[
    \cE_{2^{-m}}(\mu,\nu) \sim_{d}  \|\mu^{(m)}*\nu^{(m)}\|_2^2.
  \]
\end{lemma}

\subsection{Application to the Fractal Uncertainty Principle}

One particular problem in which additive energy in the fractal setting is relevant is the \emph{Fractal Uncertainty Principle} (FUP). Given a small scale $h\in (0,1)$, let
\[
  \mathscr{F}_h f(x) = (2\pi h)^{-d/2} \int e^{i\frac{1}{h}x\cdot y}\, f(y)\,dy
\]
be the semiclassical Fourier transform on $\R^d$. A FUP is a non-trivial estimate of the form
\[
  \| 1_X \mathscr{F}_h 1_Y  \|_{L^2(\R^d)\to L^2(\R^d)} \le h^{\beta} \quad \text{as }h\to 0,
\]
under suitable assumptions on the (typically Cantor-type) sets $X,Y\subset\R^d$ defined at resolution $h$. Here the ``trivial'' estimate is
\begin{equation}
  \label{eq:trivial-FUP}
  \| 1_X \mathscr{F}_h 1_Y \|_{L^2(\R^d)\to L^2(\R^d)} \le \min \big\{1, h^{d/2} |X|_{h}^{1/2} |Y|_{h}^{1/2}  \big\},
\end{equation}
see e.g. \cite[Before Eq. (2.7) and Eq. (6.3)]{Dyatlov19}. Note that a FUP for $X, Y$ says that $f$ cannot be supported on $Y$ in physical space, and simultaneously on $X$ in frequency space. See \cite{Dyatlov19} for an introduction to the subject, and \cite{HanSchlag20,CladekTao21,BLT23,Cohen23} for recent progress on the FUP in dimension $d\ge 2$.

It was shown by S.~Dyatlov and J.~Zahl \cite[Theorem 4.2]{DyatlovZahl16} (see also \cite[Proposition 5.4]{Dyatlov19}) that additive energy estimates imply FUP estimates near the critical threshold $|X|_{h}|Y|_{h}=h^{-d}$. While the results of \cite{DyatlovZahl16,Dyatlov19} are stated in the context of Ahlfors-regular sets in the line, the proof of  \cite[Proposition 5.4]{Dyatlov19} yields the following general estimate:
\begin{prop} \label{prop:FUP-additive-energy}
  Let $h$ be a small dyadic number. Let $X,Y\subset [0,1]^d$ be $h$-sets, and suppose
  \[
    \cE(X,X) \le h^{\sigma} |X|_h^3
  \]
  for some $\sigma>0$. Then, denoting $Z_h=Z+[0,h]$,
  \[
    \| 1_{X_h} \mathscr{F}_h 1_{Y_h} \|_{L^2(\R^d)\to L^2(\R^d)} \le h^{\beta},
  \]
  for
  \[
    \beta = \frac{3}{8}\left(d-\frac{\log(|X|_h|Y|_h)}{\log h}\right) + \frac{\sigma}{8}.
  \]
\end{prop}
Note that Proposition \ref{prop:FUP-additive-energy} improves upon the trivial estimate \eqref{eq:trivial-FUP} whenever
\[
  h^{\sigma/3} <  h^{d}|X|_{h}^{-1}|Y|_{h}^{-1} < h^{-\sigma}.
\]
It then follows from the discussion in  \S\ref{subsec:energy} that Proposition \ref{prop:FUP-additive-energy} gives a FUP near the critical parameter for wide classes of sets, for example whenever $X$ is covered by Proposition \ref{prop:application-flattening}.

We remark that A.~Cohen \cite{Cohen23} proved that, in the regime $|X|_{h} |Y|_{h} \ge h^{-d}$, the FUP holds whenever $X$ is porous on lines and $Y$ is $d$-porous  between scales $h$ and $1$. Porosity on lines in the sense of \cite{Cohen23} is slightly stronger than $1$-porosity in the sense of Definition \ref{def:porous}. Proposition \ref{prop:application-flattening} applies to more general pairs $X,Y$, since $X$ is allowed to contain lines as long as $|X|_h > h^{-1-\lambda}$ for some $\lambda>0$, and $Y$ is arbitrary. Of course, the price to pay is that the FUP is only established very close to the critical value.

%\bibliographystyle{plain}
%\bibliography{inversethm}

\end{document}